\documentclass[11pt]{article}
\usepackage{amsfonts,epsf,amsmath,amssymb}
\usepackage{epsfig}
\usepackage{latexsym}
\usepackage{pgf,tikz}

\newtheorem{theorem}{\bf Theorem}[section]

\newtheorem{proposition}[theorem]{\bf Proposition}

\newtheorem{definition}[theorem]{\bf Definition}

\newcommand{\proof}{\noindent{\bf Proof.\ }}
\newcommand{\qed}{\hfill $\square$ \bigskip}

\begin{document}
\title{Non covered vertices in Fibonacci cubes by a maximum set of disjoint hypercubes}
\author{
Michel Mollard\footnote{Institut Fourier, CNRS Universit\'e Grenoble Alpes, email: michel.mollard@univ-grenoble-alpes.fr}
}
\date{\today}
\maketitle

\begin{abstract}
\noindent The {\em Fibonacci cube} of dimension $n$, denoted as $\Gamma_n$,  is the subgraph of  $n$-cube $Q_n$ induced by vertices with no consecutive 1's. In this short note  we give an immediate proof that asymptotically all vertices of $\Gamma_n $ are covered by a maximum set of disjoint subgraphs isomorphic to $Q_k$, answering an open problem proposed in \cite{GMSZ} and solved with a longer proof in \cite{SE}. 
\end{abstract}

\noindent
{\bf AMS Subj. Class. (2010)}: 

\section{Introduction}
Let $n$ be a positive integer and denote $[n]=\{1,\ldots, n\}$, and $[n]_0=\{0,\ldots, n-1\}$. 
The {\em $n$-cube}, denoted as $Q_n$, is the graph with vertex set$$V(Q_n)=\{x_1x_2\ldots x_n \, | \, x_i\in [2]_0 ~\text{for}~ i\in [n]\}\,,$$ where 
two vertices 
are adjacent in $Q_n$ if the corresponding strings differ in exactly one position. The {\em Fibonacci $n$-cube}, denoted by $\Gamma_n$, is the subgraph of $Q_n$ induced by vertices with no consecutive 1's. Let $\{F_n\}$ be the {\em Fibonacci numbers}: $F_0 = 0$, $F_1=1$, $F_{n} = F_{n-1} + F_{n-2}$ for $n\geq 2$. The number of vertices of $\Gamma_n$ is $|V(\Gamma_n)| = F_{n+2}$.
Fibonacci cubes have been investigated from many points of view and we refer to the survey \cite{survey} for more information about them. Let $q_k(n)$ be the maximum number of disjoint subgraphs isomorphic to $Q_k$ in $\Gamma_n$. This number is studied in a recent paper \cite{GMSZ}. The authors obtained the following recursive formula

\begin{theorem}
		\label{th-q_k,n}
		For every $k\geq 1$ and $n\geq 3$
				$q_{k}(n) = q_{k-1}(n-2) + q_{k}(n-3)$.
\end{theorem}
In \cite{SE} Elif Sayg{\i} and \" Omer E\v gecio\v glu, solved an open problem proposed by the authors of \cite{GMSZ}. They proved that asymptotically all vertices of $\Gamma_n $ are covered by a maximum set of disjoint subgraphs isomorphic to $Q_k$ thus that 
\begin{theorem}
		\label{th-mainresult}
For every $k\geq 1$,  $\lim_{n\to \infty} \frac{	q_{k}(n) }{|V(\Gamma_n)|} =\frac{1}{2^k}$.
\end{theorem}
The ingenious, but long, proof they proposed is a nine cases study of the decomposition of the generating function of $q_{k}(n)$. 
The purpose of this short note is to deduce from Theorem \ref{th-q_k,n} a recursive formula for the number of non covered vertices by a maximum set of disjoint hypercubes. We obtain as a consequence an immediate proof of Theorem \ref{th-mainresult}.
 \section{Number of non covered vertices}
\begin{definition}
Let $\{P_k(n)\}^\infty_{k=1}$ be the family of sequences of integers defined by\\ 
(i)$P_k(n+3)=P_k(n)+2P_{k-1}(n+1)$ for $k\geq2$ and $n\geq 0$\\
(ii)$P_k(0)=1$,$P_k(1)=2$,$P_k(2)=3$, for $k\geq2$\\
(iii)$P_1(n)=0$ if $n\equiv1[3]$ and $P_1(n)=1$ if $n\equiv0[3]$ or $n\equiv2[3]$.\\
\end{definition}
Solving the recursion consecutively for the first values of $k$ and each class of $n$ modulo 3 we  obtain the first values of $P_k(n)$.
		\begin{table}[htb]
			\begin{center}
				
				\begin{tabular}{|c|c c c|}
			     	\hline
					\textit{n mod 3} & 0 & 1& 2 \\\hline
					$P_{1}(n)$ & 1 & 0 & 1 \\
					$P_{2}(n)$ & 1 & $\frac{2}{3}n+\frac{4}{3}$&$ \frac{2}{3}n+\frac{5}{3}$ \\
					$P_{3}(n)$ & $\frac{2}{9}n^2+\frac{2}{3}n+1$ & $\frac{2}{9}n^2+\frac{8}{9}n+\frac{8}{9}$&$ \frac{2}{3}n+\frac{5}{3}$ \\
					$P_{4}(n)$ & $\frac{4}{81}n^3+\frac{2}{9}n^2+\frac{2}{9}n+1$ & $\frac{2}{9}n^2+\frac{8}{9}n+\frac{8}{9}$& $\frac{4}{81}n^3+\frac{4}{27}n^2+\frac{10}{27}n+\frac{103}{81}$ \\\hline
				\end{tabular}
			\caption{$P_k(n)$ for $k=1,\dots,4$}
				\label{tab:initial}
			\end{center}
		\end{table}
\begin{proposition}
Let $n=3p+r$ with $r=0,1$ or $2$. For a fixed $r$, $P_k(n)$ is a polynomial in $n$ of degree at most $k-1$.
\end{proposition}
\begin{proof}
From (i) we can write $$P_k(n)=2\sum_{i=0}^{p-1}P_{k-1}(n-2-3i)+P_k(r).$$
For any integer $d$ the classical Faulhaber's formula expresses the sum  $\sum_{m=0}^{n}m^d$ as a polynomial in $n$ of degree $d+1$. Thus if $Q(n)$ is a polynomial of degree at most $d$ then  $\sum_{m=0}^{n}Q(m)$ is a polynomial in $n$ of degree at most $d+1$. Let $Q'(m)=Q(m)$ if $m\equiv0[3]$ and $0$ otherwise. Applying this to $Q'$ we obtain that $\sum_{m=0,m\equiv 0[3]}^{n}Q(m) $ is also a polynomial in $n$ of degree at most $d+1$.
Thus if  $P_{k-1}(n)$ is a polynomial in $n$ of degree at most $k-2$ then $\sum_{i=0}^{p-1}P_{k-1}(n-2-3i)$ is a polynomial of degree at most $k-1$. Since for a fixed $r$ $P_1(n)$ is a constant, by induction on $k$, $P_k(n)$ is a polynomial in $n$ of degree at most $k-1$.
\end{proof}
\qed
\begin{theorem}The number of non covered vertices of $\Gamma_n$ by $q_k(n)$ disjoint $Q_k$'s is  $P_k(n)$.\\
\end{theorem}
\proof{}
This is true for $k=1$ since  the Fibonacci cube $\Gamma_n$ has a perfect matching for $n \equiv1[3]$ and a maximum matching missing a vertex otherwise.\\
For $k>1$ this is true for $n=0,1,2$ since the values of $P_k(n)$ are respectively 1,2,3 thus are equal to $|V(\Gamma_n)|$ and there is no $Q_k$ in $\Gamma_n $.\\
Assume the property is true for some $k\geq1$ and any $n$. Then consider $k+1$. By induction on $n$ we can assume that the property is true for $\Gamma_{n-3} $. Let us prove it for $\Gamma_{n} $.\\
From Theorem \ref{th-q_k,n} we have $q_{k+1}(n)=q_{k}(n-2)+q_{k+1}(n-3)$.\\
Thus the number of  non covered vertices of $\Gamma_{n}$ by $q_{k+1}(n)$ disjoint $Q_{k+1}$'s is $$|V(\Gamma_n)|-2^{k+1}q_{k+1}(n)= F_{n+2}-2^{k+1}[q_{k}(n-2)+q_{k+1}(n-3)]= F_{n+2}-2\cdot2^{k}q_{k}(n-2)-2^{k+1}q_{k+1}(n-3).$$
Using  equalities $P_k(n-2)=F_{n}-2^kq_k(n-2)$ and $P_{k+1}(n-3)=F_{n-1}-2^{k+1}q_{k+1}(n-3)$ we obtain 
$$|V(\Gamma_n)|-2^{k+1}q_{k+1}(n)= F_{n+2}+2(P_{k}(n-2)-F_n)+P_{k+1}(n-3)-F_{n-1}.$$
From $F_{n+2}-2F_{n}-F_{n-1}=0$ and $2 P_{k}(n-2)+P_{k+1}(n-3)=P_{k+1}(n)$ the number of  non covered vertices is $P_{k+1}(n)$. So the theorem is proved.\\
\qed

For any $k$, since the number of non covered vertices is polynomial in $n$ and $|V(\Gamma_n)|=F_{n+2}
\sim\frac{3+\sqrt{5}}{2\sqrt{5}}(\frac{1+\sqrt{5}}{2})^{n}$ we obtain, like in  \cite{SE}, that
 $$\lim_{n\to \infty} \frac{	P_{k}(n) }{|V(\Gamma_n)|}=0$$ thus $$\lim_{n\to \infty} \frac{	q_{k}(n) }{|V(\Gamma_n)|}=\frac{1}{2^k}$$ \\


\begin{thebibliography}{99}

\bibitem{survey} 
	S.~Klav\v zar, 
	Structure of Fibonacci cubes: a survey, 
	J.\ Comb.\ Optim. 25 (2011) 1--18.

 
\bibitem{GMSZ}
Sylvain Gravier, Michel Mollard, Simon \v Spacapan, Sara Sabrina Zemlji\v c,
On disjoint hypercubes in Fibonacci cubes,
 Discrete Applied Mathematics, Volumes 190--191(2015) 50-55,
 http://dx.doi.org/10.1016/j.dam.2015.03.016.
  
  \bibitem{SE}
Elif Sayg{\i}, \" Omer E\v gecio\v glu,
Counting Disjoint Hypercubes in Fibonacci cubes,
 Discrete Applied Mathematics, Volume 215 (2016) 231-237, http://dx.doi.org/10.1016/j.dam.2016.07.004
 
  




\end{thebibliography}
\end{document}